\documentclass[journal]{IEEEtran}

\input{xhc.sty}

\begin{document}

\title{Data-driven Voltage Regulation in Radial \\Power Distribution Systems}

\author{Hanchen Xu,~\IEEEmembership{Student Member,~IEEE,}
        Alejandro~D.~Dom\'{i}nguez-Garc\'{i}a,~\IEEEmembership{Member,~IEEE,}
        Venugopal~V.~Veeravalli,~\IEEEmembership{Fellow,~IEEE,}
        and Peter~W.~Sauer,~\IEEEmembership{Life~Fellow,~IEEE}% <-this % stops a space
\thanks{The authors are with the Department of Electrical and Computer Engineering at the University of Illinois at Urbana-Champaign, Urbana, IL 61801, USA. Email: \{hxu45, aledan, vvv, psauer\}@illinois.edu.}
}

\maketitle

\begin{abstract}
In this paper, we develop a data-driven voltage regulation framework for distributed energy resources (DERs) in a balanced radial power distribution system.
The objective is to determine optimal DER power injections that minimize the voltage deviations from a desirable voltage range without knowing a complete power distribution system model a priori.
The nonlinear relationship between the voltage magnitudes and the power injections in the power distribution system is approximated by a linear model, the parameters of which---referred to as the voltage sensitivities---can be computed directly using information on the topology and the line parameters.
Assuming the knowledge of feasible topology configurations and distribution line resistance-to-reactance ratios, the true topology configuration and corresponding line parameters can be estimated effectively using a few sets of measurements on voltage magnitudes and power injections.
Using the estimated voltage sensitivities, the optimal DER power injections can be readily determined by solving a convex optimization problem.
The proposed framework is intrinsically adaptive to changes in system conditions such as unknown topology reconfiguration due to its data-driven nature.
The effectiveness and efficiency of the proposed framework is validated via numerical simulations on the IEEE 123-bus distribution test feeder.
\end{abstract}


%%%%%%%%%%%%%%%%%%%%%%%%%%%%%%%%%%%%%%%%%%%%%%%%
\section{Introduction}

\IEEEPARstart{V}{oltage} regulation, which is of central importance in power distribution systems, is conventionally accomplished by devices such as voltage regulators, load tap changers, and shunt capacitros \cite{kersting2006distribution}, and more recently, by distributed energy resources (DERs) with fast-responding characteristics \cite{robbins2013two, robbins2016optimal, zhang2018distributed}.
For example, in \cite{robbins2013two}, the authors proposed a two-stage distributed architecture for voltage regulation in power distribution systems, where the required reactive power injections are determined by each local controller in the first stage, and any deficiency is compensated in the second stage by other DERs providing more reactive power so as to uniformly raise the voltage profiles across the network.
In \cite{robbins2016optimal}, the authors formulated the voltage regulation problem as a convex optimization problem leveraging some relaxation technique, by solving which the optimal set-points of DER reactive power injections can be determined.
However, most of existing voltage regulation schemes using DERs assume perfect knowledge of the power distribution system model, and thus may fail in the absence of an accurate model, which is typically the case in practice.
In addition, the power distribution system model may change from time to time due to operations such as topology reconfiguration for loss minimization or load balancing \cite{baran1989network}.
As such, voltage regulation schemes which can adapt to changes in system conditions and are robust against model errors are indeed more desirable in power distribution systems.

In situations where an accurate system model is not available, data-driven methods can be applied as an alternative.
A key idea in these methods is to approximate the relation between the outputs of interest (e.g., voltage magnitudes) and the controls (e.g., power injections) by a linear model, the parameters of which are referred to as the sensitivities, and estimate the sensitivities from the measurements \cite{chen2014measurement, zhang2017power}.
For example, this idea has been pursued in estimation of injection shifting factors and power transfer distribution factors \cite{chen2014measurement}, and loss factors \cite{xu2018frequency}.
The sensitivities have also been utilized in voltage regulation problems \cite{mugnier2016model, zhang2018data, zhang2018identification}. % li2018pmu
For example, the authors proposed ambient signal based estimation methods for voltage-var sensitivities in transimission systems in \cite{zhang2018identification}.
They further developed data-driven sequential voltage control methods based on estimated voltage-var sensitivities and have proven the effectiveness via simulations using realistic data.
The sensitivities estimated from measurement enjoy several nice properties, including adaptivity to changes in system conditions such as topology reconfigurations or parameter changes.
However, existing approaches require a significant amount of measurements in order to obtain accurate estimates of the sensitivities.
This may be feasible in transmission systems equipped with phasor measurement units, yet, it may not be practical for power distribution systems.

In this paper, we develop a data-driven voltage regulation framework for DERs in a balanced radial power distribution system, the objective of which is to determine optimal DER power injections that minimize the voltage deviations from a desirable voltage range without knowing a complete distribution system model a priori.
Specifically, we will take advantage of an approximate linear model---the so-called LinDistFlow model (see, e.g., \cite{baran1989network})---to simplify the nonlinear relationships between voltage magnitudes and power injections. 
The coefficients of the LinDistFlow model are essentially the sensitivities of the squared voltage magnitudes with respect to active and reactive power injections---referred to as the voltage sensitivities, and can be computed directly using information on the topology and the line parameters.
Assuming the knowledge of feasible topology configurations and distribution line resistance-to-reactance (``$r$-to-$x$") ratios, which are typically available and do not change during a relatively short time period, the true topology configuration and corresponding line parameters can be estimated effectively using a few sets of measurements on voltage magnitudes and power injections.
Using the estimated voltage sensitivities, the optimal DER power injections can be readily determined by solving a convex optimization problem.
Theoretical analysis shows that the voltage sensitivities of interest are easily identifiable.

Part of this work has been published in as a conference paper in \cite{xu2018voltage}. This paper extends our earlier work in \cite{xu2018voltage} in several aspects.
First, we have developed an efficient topology estimation algorithm that uses the same set of measurements as the line parameter estimator and incorporated it into the framework.
Second, we also present theoretical analysis on the identifiability of the line parameters.
Third, we have added results from extensive numerical simulations on a larger distribution test feeder.

The remainder of the paper is organized as follows. Section \ref{sec:prelim} presents the preliminaries including the power distribution system model as well as the voltage regulation problem. 
Section \ref{sec:framework} presents the details of the proposed voltage regulation framework, particularly, a voltage sensitivity estimator and a voltage controller.
The identifiability of the voltage sensitivities is analyzed in Section \ref{sec:identify}.
The effectiveness of the proposed framework is validated in Section \ref{sec:sim} through some numerical simulations. 
Section \ref{sec:con} concludes this paper.

%%%%%%%%%%%%%%%%%%%%%%%%%%%%%%%%%%%%%%%%%%%%%%%%
\section{Preliminaries} \label{sec:prelim}

In this section, we provide the power distribution system model adopted in this work. 
Subsequently, we describe the voltage regulation problem in power distribution systems.

%%%%%%%%%%%%%%%%%%%%%%%
\subsection{Power Distribution System Model}

Consider a three-phase balanced power distribution system represented by a directed graph  $\calG = (\tdcalN, \calE)$, where $\tdcalN = \{0, 1, \cdots, N\}$ is the set of buses (nodes), and $\calE = \tdcalN \times \tdcalN$ is set of distribution lines (edges).
Let $\calL = \{1, 2, \cdots, L\}$ be the index set of distribution lines. 
Each electrical line $\ell \in \calL$ is associated with $(i, j) \in \calE$, i.e., $i$ is the sending end and $j$ is the receiving end of line $\ell$, with the direction from $i$ to $j$ defined to be positive. 
Let $r_\ell$ and $x_\ell$ denote the resistance and reactance of line $\ell \in \calL$, respectively, and define $\bm{r}=[r_\ell]^\top$ and $\bm{x}=[x_\ell]^\top$.
Let $z_\ell$ denote the ``$r$-to-$x$" ratio of line $\ell$, i.e., $r_\ell/x_\ell = z_\ell$, and define $\bm{z} = [z_\ell]^\top$, $\ell \in \calL$.
Throughout the rest of the paper, we make the following assumptions:
\begin{enumerate}
	\item[\textbf{A1}.] Bus~$0$ corresponds to the substation bus and $V_0$ is a constant;
	\item[\textbf{A2}.] The power distribution system is radial and connected;
	\item[\textbf{A3}.] The power distribution system is lossless; and
	\item[\textbf{A4}.] The ``$r$-to-$x$" ratios are known.
\end{enumerate}

Let $V_i$ denote the magnitude of the voltage at bus  $i \in \tdcalN$, and define $\bm{V} = [V_i]^\top$, $i \in \calN = \{1, \cdots, N\}$.
Let $\calN^g = \{1, \cdots, n\}$ denote the index set of DERs.
In addition, let $p_i^g$ and $q_i^g$ respectively denote the active and reactive power injected by DER~$i$, and define $\bm{p}^g = [p_i^g]^\top$ and $\bm{q}^g = [q_i^g]^\top$, $i \in \calN^g$.
Similarly, let $p_i^d$ and $q_i^d$ respectively denote the active and reactive power demanded by load~$i$, and define $\bm{p}^d = [p_i^d]^\top$, and $\bm{q}^d = [q_i^d]^\top$, $i \in \calN$. 
Let $\underline{p}_i^g$ and $\overline{p}_i^g$ respectively denote the minimum and  maximum active  power that can be provided by DER~$i$, and define $\bm{\underline{p}}^g = [\underline{p}_i^g]^\top$ and $\bm{\overline{p}}^g = [\overline{p}_i^g]^\top$, $i \in \calN^g$. 
Similarly, let $\underline{q}_i^g$ and $\overline{q}_i^g$ respectively denote the minimum and maximum reactive power that can be provided by DER~$i$, and define $\bm{\underline{q}}^g = [\underline{q}_i^g]^\top$ and $\bm{\overline{q}}^g = [\overline{q}_i^g]^\top$, $i \in \calN^g$.
Let $\bm{C} \in \bbR^{N \times n}$ denote the mapping matrix between the DER indices and the buses, of which the entry at the $i^{\text{th}}$ row, $j^{\text{th}}$ column of $\bm{C}$ is $1$ if DER $j$ is at bus $i$. 
Define  $\bm{p}=[p_i]^\top = \bm{C} \bm{p}^g-\bm{p}^d$, and $\bm{q} = [q_i]^\top = \bm{C} \bm{q}^g-\bm{q}^d$. 

Let $\tdbdM=[M_{i \ell}] \in \bbR^{(N+1) \times L}$ denote the node-to-edge incidence matrix of $\calG$, with $M_{i \ell} = 1$ and $M_{j \ell} = -1$ if line $\ell$ starts from bus $i$ and ends at bus $j$, and all other entries equal to zero. 
Let $\bm{M}$ denote the $(N\times L)$-dimensional matrix that results from removing the first row in $\tdbdM$.
Under Assumption \textbf{A2}, $L=N$, and $\bm{M}$ is invertible.
Note that the topology of the power distribution system is uniquely determined by $\bm{M}$; therefore, we also refer to $\bm{M}$ as a topology configuration.
A power distribution system may be operated under various feasible topology configurations.
Let $\calM$ denote the set of feasible topology configurations of the power distribution system.
Note that each topology configuration is associated with a vector of ``$r$-to-$x$" ratios.
Let $\calZ$ denote the set of ``$r$-to-$x$" ratio vectors that correspond to $\calM$.

Let $v_i = V_i^2$, and define $\bm{v} = [v_i]^\top$, $i \in \calN$, $\tdbdv = \bm{v} - v_0 \ones_N$.
Under Assumptions \textbf{A2} and \textbf{A3}, the relation between $\bm{v}$, $\bm{p}$, and $\bm{q}$, can be captured by the so-called LinDisfFlow model as follows \cite{baran1989network}:
\begin{equation} \label{eq:LinDistFlow}
    \tdbdv = \bm{R} \bm{p} + \bm{X} \bm{q}, 
\end{equation}
where $\ones_N$ is the $N$-dimensional all-ones vector, and
\begin{align}
	\bm{R} &= 2 (\bm{M}^{-1})^\top \diag{\bm{r}} \bm{M}^{-1}, \\
	\bm{X} &= 2 (\bm{M}^{-1})^\top \diag{\bm{x}} \bm{M}^{-1},
\end{align}
where $\diag{\cdot}$ returns a diagonal matrix with the entries of the argument on its diagonals; we refer to the matrices $\bm{R}$ and $\bm{X}$ as the voltage sensitivity matrices, or simply voltage sensitivities.

%%%%%%%%%%%%%%%%%%%%%%%
\subsection{Voltage Regulation Problem} \label{sec:vcp}

The objective here is to maintain the voltage magnitude at each bus $i$, $i \in \calN$, of the power distribution system within a pre-specified interval denoted by $[\underline{V}_i, \overline{V}_i]$.
While a number of means, such as load tap changers and capacitor banks can be utilized to achieve  the aforementioned objective, it is also possible to utilize the DERs present in the power distribution system. 
In this paper, we focus solely in this later mechanism for achieving voltage regulation. 
Then, the problem is to determine the DER active and reactive power injections so that
\begin{itemize} 
\item[\textbf{[C1.]}] the active and reactive power injections from each DER $i$, $i \in \calN^g$, do not exceed its corresponding capacity limits, i.e., $\bm{\underline{p}}^g \leq \bm{p}^g \leq \bm{\overline{p}}^g$, $\bm{\underline{q}}^g \leq \bm{q}^g \leq \bm{\overline{q}}^g$; and
\item[\textbf{[C2.]}] the voltage magnitude at each bus $i$, $i \in \calN$, is within the pre-specified interval, i.e., $\underline{V}_i \leq V_i \leq \overline{V}_i$.
\end{itemize}
In addition, among all feasible values of $\bm{p}^g$ and $\bm{q}^g$, we would like to select the ones that minimize some cost function, which reflects the cost of voltage deviations as well as the cost of active and reactive power provision.

In this paper, we assume no priori information on the voltage sensitivity matrices except $\calM$ and $\calZ$.
The voltage regulation problem cannot be solved without knowing the voltage sensitivity matrices.
Therefore, we will resort to the data-driven approach to estimate voltage sensitivity matrices from measurements of voltage magnitudes and power injections.

%%%%%%%%%%%%%%%%%%%%%%%%%%%%%%%%%%%%%%%%%%%%%%%%
\section{Voltage Regulation Framework} \label{sec:framework}

In this section, we propose an adaptive data-driven framework for voltage regulation using DERs. 
We first give an overview on the framework and then present the details of the fundamental building blocks of the framework.

%%%%%%%%%%%%%%%%%%%%%%%
\subsection{Framework Overview}

The proposed voltage regulation framework consists of two components, a voltage sensitivity estimator and a voltage controller.
The interaction between the different components is illustrated via the block diagram in Fig. \ref{fig:framework}.
The estimator component contains a topology estimator that estimates the topology of the power distribution system (essentially, $\bm{M}$), and a parameter estimator that estimates the line parameters ($\bm{r}$ and $\bm{x}$), using measurements of power injections and voltage magnitudes.
The estimated voltage sensitivity matrices, $\bm{R}$ and $\bm{X}$, are computed using $\bm{M}$, $\bm{r}$, and $\bm{x}$.
After that, the estimated $\bm{R}$ and $\bm{X}$, denoted respectively by $\hat{\bm{R}}$ and $\hat{\bm{X}}$, are sent to the voltage controller.
The voltage controller then computes the set-points for the DER active and reactive power injections that minimize some cost function subject to constraints \textbf{C1} and \textbf{C2}.
The DERs will be instructed to inject active and reactive power by the amount determined by the voltage controller.
A new set of measurements will be available once the DERs have modified their power injections.
These measurements will be used by the estimator to update $\hat{\bm{R}}$ and $\hat{\bm{X}}$ so as to reflect any changes in them.
The detailed formulations for the voltage sensitivity problem and the voltage regulation problem are presented next.

\begin{figure}[!t]
\centering
\includegraphics[width=3in]{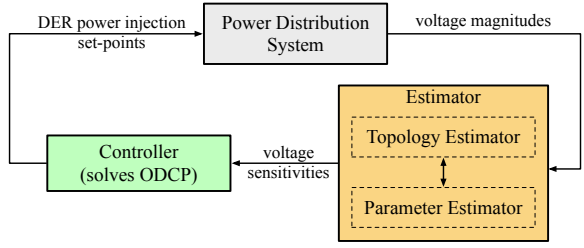}
\caption{Data-driven voltage regulation framework.}
\label{fig:framework}
\end{figure}

%%%%%%%%%%%%%%%%%%%%%%%
\subsection{Voltage Sensitivity Estimator} \label{sec:vse}

Assume at time instant $k+1$, we have measurements $V_0[k']$, $\bm{V}[k']$, $\bm{p}[k']$, $\bm{q}[k']$, $k'= 0, 1, \cdots, k$, where the index $k'$ indicates the corresponding measurement is obtained at time instant $k'$.
To reduce the computational burden, we select a subset of measurements, denoted by $\calK = \{k-m, \cdots, k\}$.
The voltage sensitivities can be estimated based on the LinDistFlow model in \eqref{eq:LinDistFlow}.
The objective of the voltage sensitivity estimator at time instant $k$ is to estimate from the measurements in $\calK$ the values of $\bm{R}$ and $\bm{X}$, which can be computed using $\bm{M}$, $\bm{r}$, and $\bm{x}$.

We propose a voltage sensitivity estimator that consists of two components, a parameter estimator and a topology estimator.
The former aims to estimate the line parameters, given the topology information, i.e., $\bm{M}$, while the later aims to determine $\bm{M}$ from $\calM$, based on the results from the parameter estimator, the details of which are presented next.

\subsubsection{Parameter estimator}

Given the topology information $\bm{M}$, to estimate $\bm{R}$ and $\bm{X}$ is essentially to estimate $\bm{r}$ and $\bm{x}$.
Let $\hat{\bm{r}}$ and $\hat{\bm{x}}$ denote the estimate of $\bm{r}$ and $\bm{x}$, respectively.
We can then formulate the parameter estimation problem by using the relation in \eqref{eq:LinDistFlow} as
\begin{align*}
	\hat{\bm{r}}, \hat{\bm{x}} = \argmin_{\bm{r}, \bm{x}} \sum_{k' \in \calK} \gamma^{k-k'}\norm{\bm{R} \bm{p}[k'] + \bm{X} \bm{q}[k'] - \tdbdv[k']}^2,
\end{align*}
subject to 
\begin{subequations} \label{eq:para_est}
	\begin{align}
		\bm{R} &= 2 (\bm{M}^{-1})^\top \diag{\bm{r}} \bm{M}^{-1}, \\
		\bm{X} &= 2 (\bm{M}^{-1})^\top \diag{\bm{x}} \bm{M}^{-1},
	\end{align}
\end{subequations}
where $\norm{\cdot}$ denotes the $L_2$-norm, $\gamma \in (0, 1]$ is a discount factor.
Essentially, the objective of the parameter estimator is to find the line parameters that can fit the LinDistFlow model best, for the given topology configuration.

We next show that \eqref{eq:para_est} has a closed-form solution.
First note that the matrix $\diag{\bm{x}}$ can be decomposed as follows:
\begin{align} \label{eq:basis_expan}
	\diag{\bm{x}} = \sum_{\ell=1}^L x_\ell \bm{e}_\ell \bm{e}_\ell^\top,
\end{align}
where $\bm{e}_\ell$ is the $\ell^{\mathrm{th}}$ basis vector in $\bbR^L$, i.e., all entries in $\bm{e}_\ell$ are $0$ except the $\ell^{\mathrm{th}}$ entry, which equals to $1$.
Using \eqref{eq:basis_expan}, we obtain that
\begin{align} \label{eq:Xq}
	\bm{X} \bm{q}[k'] & = 2 (\bm{M}^{-1})^\top \diag{\bm{x}} \bm{M}^{-1} \bm{q}[k'] \nonumber \\
	& = 2 (\bm{M}^{-1})^\top \sum \limits_{\ell=1}^L x_\ell \bm{e}_\ell \bm{e}_\ell^\top \bm{M}^{-1} \bm{q}[k'] \nonumber \\
	& = \sum \limits_{\ell=1}^L \bm{\Xi}_\ell \bm{q}[k'] x_\ell,
\end{align}
where $\bm{\Xi}_\ell = 2 (\bm{M}^{-1})^\top \bm{e}_\ell \bm{e}_\ell^\top \bm{M}^{-1}$.
Similarly,
\begin{align} \label{eq:Rp}
	\bm{R} \bm{p}[k'] = \sum \limits_{\ell=1}^L \bm{\Xi}_\ell \bm{p}[k'] r_\ell = \sum \limits_{\ell=1}^L \bm{\Xi}_\ell z_\ell \bm{p}[k'] x_\ell.
\end{align}

Let $\bm{\rho}_\ell[k'] = \gamma^{\frac{k-k'}{2}}(z_\ell \bm{p}[k'] + \bm{q}[k'])$, $\ell \in \calL$, and define
\begin{align} \label{eq:psi_mat}
	\bm{\Psi}[k] = 
	\begin{bmatrix}
		\bm{\Xi}_1 \bm{\rho}_1[k-m] & \cdots & \bm{\Xi}_L \bm{\rho}_L[k-m] \\
		\vdots & \vdots & \vdots \\
		\bm{\Xi}_1 \bm{\rho}_1[k] & \cdots & \bm{\Xi}_L \bm{\rho}_L[k]
	\end{bmatrix},
\end{align}
and
\begin{align} \label{eq:psi_vec}
	\bm{\psi}[k] = [\gamma^{\frac{m}{2}} \tdbdv[k-m]^\top, \cdots, \gamma^{\frac{0}{2}} \tdbdv[k]^\top]^\top.
\end{align}
Note that $\bm{\Psi}[k] \in \bbR^{(m+1)N \times L}$ and $\bm{\psi}[k] \in \bbR^{(m+1)N}$ are dependent on $\calK$.
Then \eqref{eq:para_est} can be equivalently formulated in the classical form of a linear regression problem as follows:
\begin{align} \label{eq:para_est2}
	\minimize_{\bm{x}} \norm{\bm{\Psi}[k] \bm{x} - \bm{\psi}[k]}^2,
\end{align}
the solution to which is given by
\begin{align} \label{eq:vse_sol_x}
	\hat{\bm{x}} =  \bm{\Psi}[k]^\dagger \bm{\psi}[k]
%	\hat{\bm{x}} = (\bm{\Psi}[k]^\top \bm{\Psi}[k])^{-1} \bm{\Psi}[k]^\top \bm{\psi}[k].
\end{align}
where $\bm{\Psi}[k]^\dagger$ denotes the pseudo-inverse of $\bm{\Psi}[k]$, obtained via singular value decomposition.
Note that $\bm{\Psi}[k]$ needs to have full rank, i.e., $\rank{\bm{\Psi}[k]} = L$, in order to estimate $\bm{x}$.
The resistance vector can be estimated using
\begin{align} \label{eq:vse_sol_r}
	\hat{\bm{r}} = \diag{\bm{z}} \hat{\bm{x}}.
\end{align}

Define a residual vector, denoted by $\bm{\varepsilon}$, as follows:
\begin{align} \label{eq:vareps}
	\bm{\varepsilon} = \hat{\bm{R}} \bm{p} + \hat{\bm{X}} \bm{q} - \tdbdv,
\end{align}
where 
\begin{align}
	\hat{\bm{R}} &= 2 (\bm{M}^{-1})^\top \diag{\hat{\bm{r}}} \bm{M}^{-1}, \label{eq:R_hat} \\
	\hat{\bm{X}} &= 2 (\bm{M}^{-1})^\top \diag{\hat{\bm{x}}} \bm{M}^{-1}. \label{eq:X_hat}
\end{align}
Given a set of measurements, we can compute a residual vector for each $\bm{M} \in \calM$ deterministically.
%To emphasize the fact that $\bm{\varepsilon}$ is a function of $\bm{M}$, we write $\bm{\varepsilon}(\bm{M})$.

\subsubsection{Topology estimator}

The objective of the topology estimator is to find $\bm{M} \in \calM$ such that a weighted sum of $\norm{\bm{\varepsilon}}$ over several time instants is minimized.
At time instant $k+1$, the topology estimation problem can be formulated as:
\begin{align} \label{eq:topology_est}
	\hat{\bm{M}} = \argmin_{\bm{M} \in \calM} \epsilon_{\bm{M}},
\end{align}
with
\begin{align} \label{eq:residual_error}
	\epsilon_{\bm{M}} = \sum_{k' \in \calK} \gamma^{k-k'}\norm{\bm{\varepsilon}[k']},
\end{align}
where $\bm{\varepsilon}$ is computed through \eqref{eq:vareps} to \eqref{eq:X_hat}.
We refer to $\epsilon_{\bm{M}}$ as the \textit{residual error} associated with topology configuration $\bm{M}$.

Essentially, the topology estimator selects the topology under which the residual error is minimized, where the line parameters are estimated by the parameter estimator.
The intuition here is that different topology configurations will impose different structural constraints on voltage sensitivity matrices, which consequently will impact the residual error.
The true topology configuration is expected to result in the least residual error.
The voltage sensitivity estimation algorithm is summarized in Algorithm \ref{algo:vse}.

\begin{algorithm}[!t]
    \SetAlgoLined
    \DontPrintSemicolon
    \KwData{\\~~~~~~$\calM$: set of feasible topology configurations \\~~~~~~$\calZ$: set of ``$r$-to-$x$" ratio vectors \\~~~~~~$\bm{p}[k'], \bm{q}[k'], \bm{v}[k']$: active power, reactive power, voltage magnitude measurements, $k' \in \calK$}
    \KwResult{\\~~~~~~$\hat{\bm{M}}$: estimated topology configuration \\~~~~~~$\hat{\bm{r}}, \hat{\bm{x}}$: estimated line parameters}
    \For{$\bm{M} \in \calM$, $\bm{z} \in \calZ$}{
    	Construct $\bm{\Psi}$ and $\bm{\psi}$ according to \eqref{eq:psi_mat} and \eqref{eq:psi_vec}\;
    	Compute the pseudo-inverse of $\bm{\Psi}$, i.e., $\bm{\Psi}^\dagger$\;
    	Compute line parameters using \eqref{eq:vse_sol_x} and \eqref{eq:vse_sol_r}\;
%    	\begin{align*}
%			\hat{\bm{x}} = \bm{\Psi}^\dagger \bm{\psi},
%    		\hat{\bm{r}} = \diag{\bm{z}} \hat{\bm{x}}
%    	\end{align*}\;\vspace{-0.2in}
    	Compute voltage sensitivities using \eqref{eq:R_hat} and \eqref{eq:X_hat}\;
%    	\begin{align*}
%    		\hat{\bm{R}} &= 2 (\bm{M}^{-1})^\top \diag{\hat{\bm{r}}} \bm{M}^{-1} \\
%			\hat{\bm{X}} &= 2 (\bm{M}^{-1})^\top \diag{\hat{\bm{x}}} \bm{M}^{-1}
%    	\end{align*}\;\vspace{-0.2in}
    	Compute the residual error via \eqref{eq:residual_error}\;
%    	\begin{align*}
%    		\epsilon_{\bm{M}} = \sum_{k' \in \calK} \gamma^{k-k'} \norm{\hat{\bm{R}} \bm{p}[k'] + \hat{\bm{X}} \bm{q}[k'] - \tdbdv[k']}^2
%    	\end{align*}\;\vspace{-0.2in}
    }
    Select topology configuration $\hat{\bm{M}}$ using \eqref{eq:topology_est}
%    \begin{align*}
%    	\hat{\bm{M}} = \argmin_{\bm{M} \in \calM} \epsilon_{\bm{M}}
%    \end{align*}
    and line parameters $\hat{\bm{r}}, \hat{\bm{x}}$ to be the ones associated with $\hat{\bm{M}}$
\label{algo:vse}
\caption{Voltage Sensitivity Estimation}
\end{algorithm}

%%%%%%%%%%%%%%%%%%%%%%%
\subsection{Voltage Controller}

The voltage controller aims to determine the optimal set-points for the DER active and reactive power injections while meeting all requirements discussed in Section \ref{sec:vcp}.
Note that for a given set of power injections, the resulting voltage magnitude at each bus can be estimated using \eqref{eq:LinDistFlow}, where $\hat{\bm{R}}$ and $\hat{\bm{X}}$ are used instead of $\bm{R}$ and $\bm{X}$.
Define $\underline{\bm{v}} = [\underline{V}_i^2]^\top$ and $\overline{\bm{v}} = [\overline{V}_i^2]^\top$, $i \in \calN$.
Then, the voltage control problem can be formulated as the following convex program:
\begin{equation*} 
\minimize_{\bm{p}^g, \bm{q}^g} ~ c(\bm{p}^g, \bm{q}^g)
\end{equation*}
subject to 
\begin{subequations} \label{eq:vcp}
	\begin{align} \label{eq:vcp_c1}
		\bm{v} = \hat{\bm{R}} (\bm{C} \bm{p}^g - \bm{p}^d) + \hat{\bm{X}} (\bm{C} \bm{q}^g - \bm{q}^d) + v_0 \ones_N,
	\end{align}
	\begin{align}
		\bm{\underline{p}}^g \leq \bm{p}^g \leq \bm{\overline{p}}^g,
		\bm{\underline{q}}^g \leq \bm{q}^g \leq \bm{\overline{q}}^g,
	\end{align}
\end{subequations}
with
\begin{align*}
	c(\bm{q}^g) = & (\bm{p}^g)^\top \bm{W}^p \bm{p}^g + (\bm{q}^g)^\top \bm{W}^q \bm{q}^g \nonumber \\
	& + \beta_1 \norm{[\underline{\bm{v}} - \bm{v}]_+}^2 + \beta_2 \norm{[\bm{v} - \overline{\bm{v}}]_+}^2,
\end{align*}
where $\bm{W}^p = \diag{w_1^p, \cdots, w_n^p}$, $\bm{W}^q = \diag{w_1^q, \cdots, w_n^q}$ are non-negative diagonal matrices, $[~]_+$ returns a non-negative vector, $\beta_1$ and $\beta_2$ are non-negative weights.
The first two terms of $c(\cdot)$ are the cost of active and power injections, while the last two terms penalize the violation of constraint \textbf{C2}.

Constraint \eqref{eq:vcp_c1} is the LinDistFlow model, which is used to predict the voltage magnitudes for given power injections.
Note that $\bm{p}^d$ and $\bm{q}^d$ are measured before solving the voltage control problem.
Solving \eqref{eq:vcp} gives the optimal set-points for the DER active and reactive power injections.

%%%%%%%%%%%%%%%%%%%%%%%%%%%%%%%%%%%%%%%%%%%%%%%%
\section{Voltage Sensitivity Identifiability} \label{sec:identify}

In this section, we first introduce the path matrix associated with a graph and then analyze the conditions under which the line parameters and correspondingly voltage sensitivities, can be estimated.

%%%%%%%%%%%%%%%%%%%%%%%
\subsection{Path Matrix}

Let $\calP_i \subseteq \calL$ denote the set of lines that form the path from bus $0$---referred to as the root---to bus $i$.
Since the power distribution system is radial, then $\calP_i$ is unique (see Theorem 2.1.4 in \cite{west2001introduction}).
Bus $i$ is a leaf if there is no line that starts from it.
We say bus $i$ is closer to the root than bus $j$ if $|\calP_i| < |\calP_j|$, where $|\cdot|$ denotes the cardinality of a set.
Let $\bm{P} = [P_{\ell i}] \in \bbR^{L \times N}$ denote the path matrix of $\calG$, with $P_{\ell i} = 1$ ($P_{\ell i} = -1$) if line $\ell$ is on $\calP_i$ and their directions agree (disagree), and all other entries equal to zero.
We choose the sending end of line to be the bus that is closer to the root, then all entries in $\bm{P}$ are in $\{0, 1\}$ by definition since the direction of $\calP_i$ and any line on it will always agree.
Under this setup, the relation between $\bm{P}$ and $\bm{M}$ is given by the following lemma.
\begin{lemma}
	$\bm{P}$ and $\bm{M}$ are related by $\bm{M}^{-1} = -\bm{P}$. (see also Theorem 2.10 in \cite{bapat2010graphs}.)
\end{lemma}
\begin{proof}
	Consider the entry at the $i^{\text{th}}$ row and $j^{\text{th}}$ column in $\bm{M} \bm{P}$, which is $\sum_{\ell=1}^L M_{i \ell} P_{\ell j}$.
	\begin{enumerate}
		\item Consider first the case where $i=j$. If line $\ell$ is not connected to bus $i$, then $M_{i\ell} = 0$. 
			If line $\ell$ starts from bus $i$, then $M_{i\ell} = 1$ and $P_{\ell i} = 0$. If line $\ell$ ends at bus $i$, then $M_{i\ell} = -1$ and $P_{\ell i} = 1$. 
			Obviously, there is one line that ends at bus $i$. 
			Moreover, such a line is unique since otherwise there will be two paths from the root to bus $i$.
			Therefore, $\sum_{\ell=1}^L M_{i \ell} P_{\ell i} = M_{i \ell_i} P_{\ell_i i} = -1$, where line $\ell_i$ is the line that ends at $i$.
		\item Next consider the case where $i \neq j$. Similar to the previous case, we only need to consider the lines that starts from or ends at bus $i$.
			\begin{enumerate}
				\item If line $\ell$ ends at bus $i$, then $M_{i\ell} = -1$. If $\ell \notin \calP_j$, then $P_{\ell j} = 0$ and $M_{i \ell} P_{\ell j} = 0$. 
					If $\ell \in \calP_j$, then $P_{\ell j} = 1$. In the latter case, there must exist a unique line $\ell' \in \calP_j$ that starts from $i$. Then $M_{i \ell} P_{\ell j} + M_{i \ell'} P_{\ell' j} = -1 + 1 = 0$. Therefore, $\sum_{\ell=1}^L M_{i \ell} P_{\ell j} = 0$.
				\item If line $\ell$ starts from bus $i$, then $M_{i\ell} = 1$. If $\ell \notin \calP_j$, then $P_{\ell j} = 0$ and $M_{i \ell} P_{\ell j} = 0$. If $\ell \in \calP_j$, then $P_{\ell j} = 1$. In the latter case, there must exist a unique line $\ell' \in \calP_j$ that ends at $i$. Similar to the previous argument, $\sum_{\ell=1}^L M_{i \ell} P_{\ell j} = 0$.
			\end{enumerate}
	\end{enumerate}
	In sum, $\sum_{\ell=1}^L M_{i \ell} P_{\ell j}$ equals to $1$ if $i=j$ and $0$ otherwise; therefore, $\bm{M}^{-1} = -\bm{P}$.
\end{proof}
The path matrix will play an important role in the analysis of the identifiability of the voltage sensitivities, which is to be detailed in the next section.

%%%%%%%%%%%%%%%%%%%%%%%
\subsection{Identifiability Analysis}

Before presenting the main results for the identifiability of voltage sensitivities, we introduce the concept of downstream buses.
\begin{definition}
	If line $\ell \in \calP_i$, $\ell \in \calL$, i.e., line $\ell$ is on the path from the root to bus $i$, then bus $i$ is a downstream bus of line $\ell$. The set of downstream buses of line $\ell$ is denoted by $\calN_\ell$.
\end{definition}

As discussed in Section \ref{sec:vse}, $\bm{\Psi}[k]$ needs to have full rank, i.e., $\rank{\bm{\Psi}[k]} = L$, in order to estimate $\bm{x}$ according to \eqref{eq:vse_sol_x}.
When $\bm{\Psi}$ does not have full rank, some of the line parameters cannot be estimated from the measurements.
The main results for the  voltage sensitivity identifiability is stated as follows:
\begin{theorem} \label{thm1}
	The parameter of line $\ell$, $\ell \in \calL$, is identifiable if and only if the following condition is satisfied for some $k' \in \calK$:
	\begin{align} \label{eq:cond}
		\sum_{i \in \calN_\ell} z_{\ell} p_i[k'] + q_i[k'] \neq 0.
	\end{align}
\end{theorem}
\begin{proof}
	Using the path matrix, $\bm{\Xi}_\ell$ can be written as $\bm{\Xi}_\ell = 2 \bm{P}^\top \bm{e}_\ell \bm{e}_\ell^\top \bm{P}$.
	Note that $\bm{P}^\top \bm{e}_\ell$ is the $\ell^{\mathrm{th}}$ column of $\bm{P}^\top$ and $\bm{\Xi}_\ell$ is a rank-one matrix.
	Let $\bm{\pi}_\ell = \bm{P}^\top \bm{e}_\ell$, then $\bm{P}^\top = [\bm{\pi}_1, \cdots, \bm{\pi}_L]$.
	Then,
	\begin{align}
		\bm{\Xi}_\ell = 2 \bm{P}^\top \bm{e}_\ell \bm{e}_\ell^\top \bm{P} = 2 \bm{\pi}_\ell \bm{\pi}_\ell^\top,
	\end{align}
	and $\bm{\Psi}[k]$ can be written as
	\begin{align} \label{eq:psi_mat2}
		\bm{\Psi}[k] 
		&= 2
		\begin{bmatrix}
			\bm{\pi}_1 \bm{\pi}_1^\top \bm{\rho}_1[k-m] & \cdots & \bm{\pi}_L \bm{\pi}_L^\top \bm{\rho}_L[k-m] \\
			\vdots & \vdots & \vdots \\
			\bm{\pi}_1 \bm{\pi}_1^\top \bm{\rho}_1[k] & \cdots & \bm{\pi}_L \bm{\pi}_L^\top \bm{\rho}_L[k]
		\end{bmatrix}.% \nonumber \\
%		&=
%		\begin{bmatrix}
%			\bm{P}^\top \diag{\bm{\pi}_1^\top \bm{\rho}_1[k_1], \cdots, \bm{\pi}_L^\top \bm{\rho}_L[k_1]} \\
%			\vdots \\
%			\bm{P}^\top \diag{\bm{\pi}_1^\top \bm{\rho}_1[k_m], \cdots, \bm{\pi}_L^\top \bm{\rho}_L[k_m]}
%		\end{bmatrix}.
	\end{align}
	
	Let $\calL = \calL_1 \cup \calL_0$, where $\calL_1$ and $\calL_0$ are the sets of lines that meet and do not meet the conditions in \eqref{eq:cond}, respectively.
	If line $\ell \in \calL_0$, then $\forall k' \in \calK$,
	\begin{align} \label{eq:cond1}
		\sum_{i \in \calN_\ell} z_{\ell} p_i[k'] + q_i[k'] = 0.
	\end{align}
	
	Note that the $i^{\mathrm{th}}$ entry in $\bm{\pi}_\ell$ is $1$ if and only if bus $i$ is a downstream bus of line $\ell$.
	Essentially, the non-zero entries in $\bm{\pi}_\ell$, which are ones, indicate the downstream buses of line $\ell$.
	Therefore, it follows from \eqref{eq:cond1} that, $\forall k' \in \calK$:
	\begin{align} \label{eq:cond2}
		\bm{\pi}_\ell^\top \bm{\rho}_\ell[k'] = 0.
	\end{align}
	Consequently, all entries in the $\ell^{\mathrm{th}}$ column of $\bm{\Psi}[k]$ are zero, and the value of $x_\ell$ does not affect the objective function in \eqref{eq:para_est2}.
	Under such condition, $x_\ell$ cannot be identified.
	For all line $\ell \in \calL_0$, we can remove the $\ell^{\mathrm{th}}$ column of $\bm{\Psi}[k]$, the $\ell^{\mathrm{th}}$ entry of $\bm{x}$ and $\bm{\psi}[k]$, and obtain an estimation problem of reduced size.
	
	Next we show that the line parameter can be identified as long as condition \eqref{eq:cond} is satisfied.
	Without loss of generality, we assume $\calL_1 = \calL$ since otherwise we can remove the unidentifiable variables to obtain a reduced problem that satisfied this condition.
	Then, \eqref{eq:cond2} is satisfied for all $\ell \in \calL$ and for some $k' \in \calK$.
	Assume $\rank{\bm{\Psi}[k]} < L$, then there exist $a_1, \cdots, a_L \in \bbR$, which are not all zero, such that
	\begin{align} \label{eq:lin_depen}
		\bm{\Psi}[k] [a_1, \cdots, a_L]^\top = \zeros_{mL},
	\end{align}
	where $\zeros_L$ is an $mL$-dimensional all-zeros vector.
	Without loss of generality, assume $a_1, \cdots, a_{L'}$ are not zero, while $a_{L'+1}, \cdots, a_L$ are all zero, where $1 < L' \leq L$.
	Then, it follows from \eqref{eq:psi_mat2} and \eqref{eq:lin_depen} that 
	\begin{align}
		\sum_{l=1}^{L'} a_l \bm{\pi}_l^\top \bm{\rho}_l[k'] \bm{\pi}_l = \zeros_L.
	\end{align}
	Since $\bm{\pi}_1, \cdots, \bm{\pi}_{L'}$ are linear independent, then $a_l \bm{\pi}_l^\top \bm{\rho}_l[k'] = 0$ for $\ell = 1, \cdots, L'$. 
	However, since for any $\ell \in \calL$ there exists some $k' \in \calK$ such that $\bm{\pi}_\ell^\top \bm{\rho}_\ell[k'] = 0$, then $a_\ell = 0$ for $\ell = 1, \cdots, L'$, which leads to a contradiction.
	Therefore, $\rank{\bm{\Psi}[k]} = L$ and the line parameters can be identified.
\end{proof}
\begin{remark}
	The voltage sensitivity matrices can be readily computed if all line parameters can be identified. 
	If some line parameter cannot be identified, the resulting voltage sensitivity matrices may not be accurate.
	This, however, will not have any impact on estimating voltage magnitudes since in such cases the line parameter does not affect the voltage magnitudes anyway.
	Specifically, if follows from \eqref{eq:LinDistFlow}, \eqref{eq:Xq}, and \eqref{eq:Rp} that 
	\begin{align}
		\tdbdv[k'] = \bm{R} \bm{p}[k'] + \bm{X} \bm{q}[k'] &= 2 \sum \limits_{\ell=1}^L \bm{\pi}_\ell \bm{\pi}_\ell^\top \bm{\rho}[k'] x_\ell,
	\end{align}
	in which $\bm{\pi}_\ell^\top \bm{\rho}[k'] = 0$ if $x_\ell$ cannot be identified.
	Therefore, for the purpose of achieving voltage control, the proposed voltage sensitivity estimation algorithm is effective.
\end{remark}

If we think of $z_{\ell} p_i[k'] + q_i[k']$ as some ``combined power" (in the sense that it is a combination of active and reactive power), then \eqref{eq:cond} essentially indicates that the sum of combined power injection at all downstream buses of line $\ell$ is nonzero, or equivalently, there exists some combined power flow on line $\ell$.
For any line whose receiving end is a leaf, its parameter can be identified as long as the combined power injection at the receiving end is nonzero.
Condition \eqref{eq:cond} can be easily satisfied in actual power distribution systems.

%%%%%%%%%%%%%%%%%%%%%%%%%%%%%%%%%%%%%%%%%%%%%%%%
\section{Numerical Simulations} \label{sec:sim}

\begin{figure}[!t]
\centering
\includegraphics[width=3.5in]{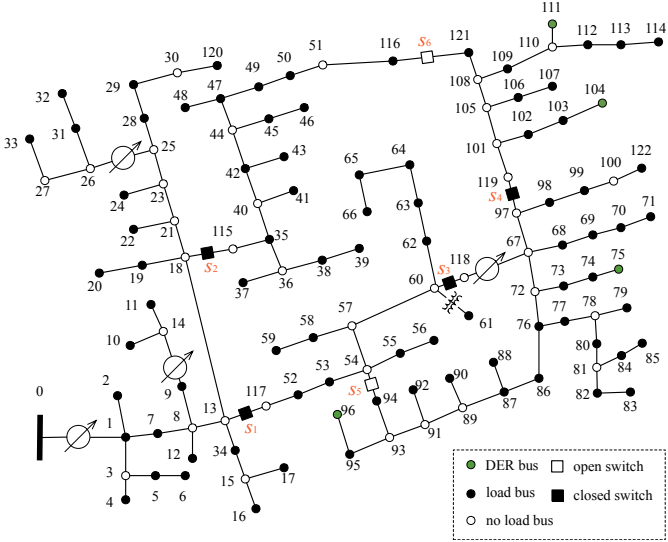}
\caption{IEEE 123-bus distribution test feeder.}
\label{fig:123bus}
\end{figure}

In this section, we validate the effectiveness of the proposed framework using a modified single-phase IEEE 123-bus distribution test feeder from \cite{test_feeder}, the topology of which is shown in Fig. \ref{fig:123bus}.
There are six switches in this feeder, four of which are normally closed and the other two open so as to maintain a radial structure of the system.
Under Assumption \textbf{A2}, this feeder has nine possible topology configurations as listed in Table \ref{table:topology_config}, among which configuration $0$ is the nominal one.

\begin{table}[!t]
	\caption{Switch Status Under Feasible Topology Configurations.}
	\label{table:topology_config}
	\centering
	\begin{tabular}{ccccccccc}
		\toprule
		config. & $s_1$ & $s_2$ & $s_3$ & $s_4$ & $s_5$ & $s_6$\\
		\midrule
		0 & on & on & on & on & off & off\\
		1 & on & on & off & on & on & off\\
		2 & on & on & on & off & off & on\\
		3 & on & on & off & on & off & on\\
		4 & on & off & on & on & off & on\\
		5 & off & on & on & on & off & on\\
		6 & on & on & off & off & on & on\\
		7 & on & off & off & on & on & on\\
		8 & off & on & off & on & on & on\\
		\bottomrule
	\end{tabular}
\end{table}

%\begin{table}[!t]
%	\caption{Feasible Topology Configurations.}
%	\label{table:topology_config}
%	\centering
%	\begin{tabular}{ccccccccc}
%		\toprule
%		config. & $s_1$ & $s_2$ & $s_3$ & $s_4$ & $s_5$ & $s_6$\\
%		\midrule
%		0 & on & on & on & on & off & off\\
%%		1 & on & on & on & off & on & off\\
%		2 & on & on & off & on & on & off\\
%%		3 & on & off & on & on & on & off\\
%%		4 & off & on & on & on & on & off\\
%		5 & on & on & on & off & off & on\\
%		6 & on & on & off & on & off & on\\
%		7 & on & off & on & on & off & on\\
%		8 & off & on & on & on & off & on\\
%		9 & on & on & off & off & on & on\\
%%		10 & on & off & on & off & on & on\\
%%		11 & off & on & on & off & on & on\\
%		12 & on & off & off & on & on & on\\
%		13 & off & on & off & on & on & on\\
%%		14 & off & off & on & on & on & on\\
%		\bottomrule
%	\end{tabular}
%\end{table}

The loads are simulated in the following way. 
First, historical hourly loads of a residential building in San Diego \cite{load_data} are interpolated to increase the time granularity to $1$ second.
A zero-mean Gaussian noise with a standard deviation of $0.01$~p.u. is also added to the interpolated loads, which are then scaled to match the active and reactive power load levels in the feeder.
Four DERs are added at buses 76, 97, 105, 112, respectively, with reactive power outputs within $[-200, 200]$~kVAr.
We set $w_i^p = 1 + 0.1i$ and $w_i^q = 1 + 0.1i$, for $i \in \calN^g$.
For simplicity, we assume the DERs do not output any active power, i.e., $\bm{\underline{p}}^g = \bm{\overline{p}}^g = \zeros_n$.
The minimum and maximum voltage magnitudes are $0.95$ p.u. and $1.05$ p.u., respectively.
In addition, $\beta_1 = \beta_2 = 1\times 10^5$.
Unless otherwise specified, the discounter factor $\gamma$ is set to $1$, and the underlying topology configuration is configuration $0$, i.e., the nominal one.
While we assume the power distribution system is lossless for the analysis, in the simulation, we use a full nonlinear power flow model that is solved using Matpower \cite{zimmerman2011matpower}.

\subsection{Estimation Accuracy}

Throughout this part, the DERs do not inject any reactive power into the power distribution system.

\subsubsection{Noise-free case}

\begin{figure}[!t]
\centering
\includegraphics[width=3.5in]{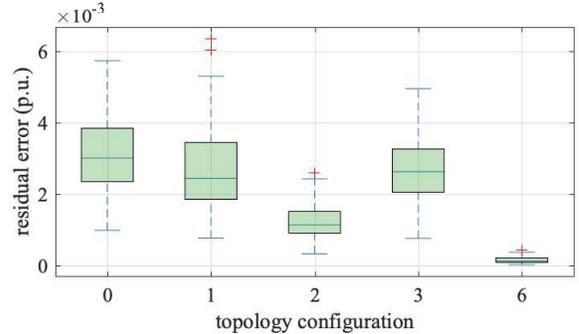}
\caption{Residual errors under different topology configurations with $10$ sets of noise-free measurements.}
\label{fig:noise_free_residual}
\end{figure}

We first evaluate the accuracy of proposed estimation algorithm in the case where the measurements are noise-free.
The algorithm is evaluated in $100$ Monte Carol runs under various loading conditions.
In each simulation run, $10$ sets of measurements are used to compute the residual error.
Residual errors are computed for each feasible topology configuration in $\calM$, while the underlying true topology configuration is one of them.
A box-plot of residual errors associated with each topology configuration when the underlying topology configuration is configuration $6$, is shown in Fig. \ref{fig:noise_free_residual}.
Note that residual errors associated with topology configurations $4, 5, 7, 8$ are at least one order of magnitude larger than those of the other configurations, and are hence not plotted.
This is the case where the residual error differences between each topology configuration is the smallest.
Yet, it is still obvious that the true topology configuration results in the minimum residual error, which is one order of magnitude smaller than those of other configurations.

%\begin{figure*}
%\centering
%\includegraphics[width=6.5in]{noise_free_x}
%\caption{Estimated and true values of line reactances using $10$ sets of noise-free measurements.}
%\label{fig:noise_free_x}
%\end{figure*}

\begin{figure}[!t]
\centering
\includegraphics[width=3.in]{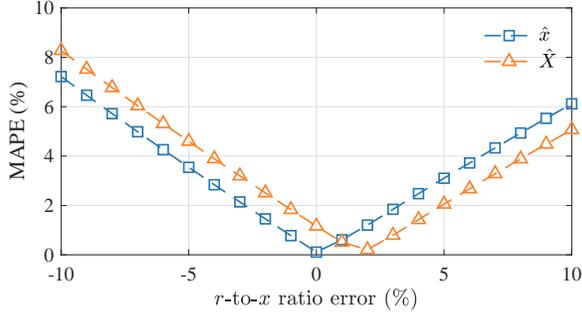}
\caption{Impacts of errors in $r$-to-$x$ ratios on parameter estimation accuracy in the noise-free case.}
\label{fig:noise_free_zeta}
\end{figure}

The parameter estimation accuracy is evaluated using the mean absolute percentage error (MAPE) of the estimates.
%Figure \ref{fig:noise_free_x} shows the estimated line reactances in one simulation run and $1$ set of measurements is utilized.
When $1$ set of measurements is utilized, a typical MAPE of $\hat{\bm{x}}$ is $0.11\%$, and that of $\hat{\bm{X}}$ is $1.16\%$, both of which are really small.
We note that the loading conditions of the power distribution system does not affect the accuracy of the proposed algorithm.
The $r$-to-$x$ ratios of all lines are assumed to be known.
Figure \ref{fig:noise_free_zeta} shows that the MAPE is almost linear with respect to the $r$-to-$x$ ratio errors.
Therefore, relatively small error in the $r$-to-$x$ ratios will not result in a significant increase in parameter estimator errors.

\subsubsection{Noisy case}

\begin{figure}[!t]
\centering
\includegraphics[width=3.5in]{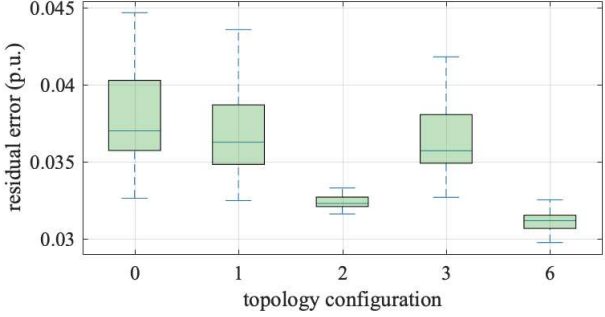}
\caption{Residual errors under different topology configurations with $60$ sets of noisy measurements..}
\label{fig:noisy_residual}
\end{figure}

To see the impacts of measurement noise, we add white Gaussian noise to measurements such that the signal-to-noise ratio (SNR) is $92$~dB, which is adopted by authors in \cite{xie2014dimensionality}.
More measurements are required to obtain a good estimation accuracy in the presence of measurement noise.
The algorithm is again evaluated in $100$ Monte Carol runs under the same setup as the noise-free case, except that $60$ sets of measurements---corresponding to measurements collected in $1$ minute---are used to compute the residual error.
A box-plot of residual errors associated with each topology configuration is shown in Fig. \ref{fig:noisy_residual}.
Note that residual errors associated with topology configurations $4, 5, 7, 8$ are one order magnitude larger than those of the other configurations, and are hence not plotted.
The true topology configuration, i.e., configuration $6$, still results in the minimum residual error.
We note that increasing more measurements generally lead to higher accuracy in identifying the topology configuration.

\begin{figure}[!t]
\centering
\includegraphics[width=3.in]{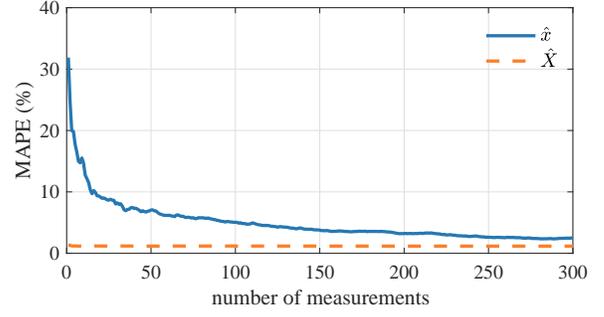}
\caption{Impacts of measurement numbers on parameter estimation accuracy in the noisy case.}
\label{fig:noisy_x}
\end{figure}

The number of measurements have a direct impacts on the estimation accuracy.
As is shown in Fig. \ref{fig:noisy_x}, the MAPE of $\hat{\bm{x}}$ drops quickly when increasing the number of measurements, approximately from $31.9\%$ with $1$ set of measurements to $2.51\%$ with $300$ sets of measurements.
The MAPE of $\hat{\bm{X}}$---which is what really matters---is relatively insensitive to the number of measurements, with an MAPE around $1.17\%$.
Indeed, this result illustrates the high effectiveness and efficiency of the proposed estimation algorithm.

\begin{figure}[!t]
\centering
\includegraphics[width=3.in]{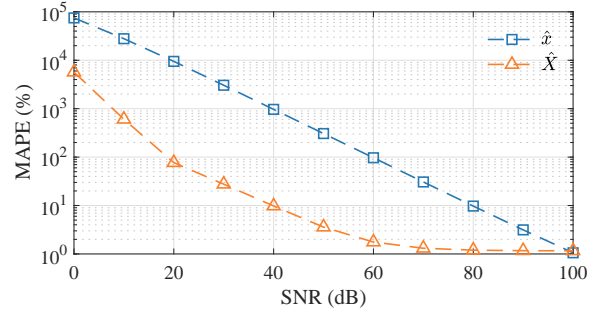}
\caption{Impacts of SNR on parameter estimation accuracy when $300$ sets of measurements are used.}
\label{fig:noisy_x_SNR}
\end{figure}

Figure \ref{fig:noisy_x_SNR} shows the impacts of SNR on parameter estimation accuracy when $300$ sets of measurements are used.
When the SNR is beyond $50$~dB, the MAPE of the voltage sensitivity matrix is within $3.6\%$, which is relatively small.
In the rest of of the simulation, we assume a SNR of $92$~dB for all measurements.

\subsubsection{Accuracy under topology reconfiguration}

\begin{figure}[!t]
\centering
\includegraphics[width=3.in]{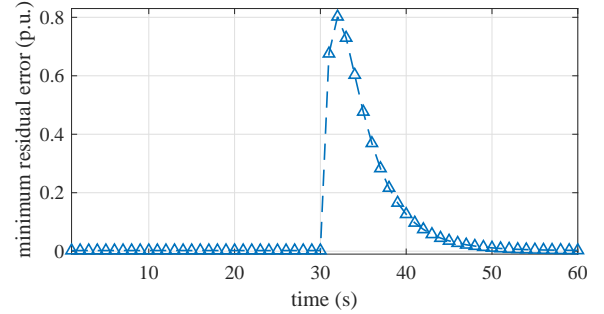}
\caption{Minimum residual error under topology reconfiguration.}
\label{fig:noisy_reconfig_residual}
\end{figure}

\begin{figure}[!t]
\centering
\includegraphics[width=3.in]{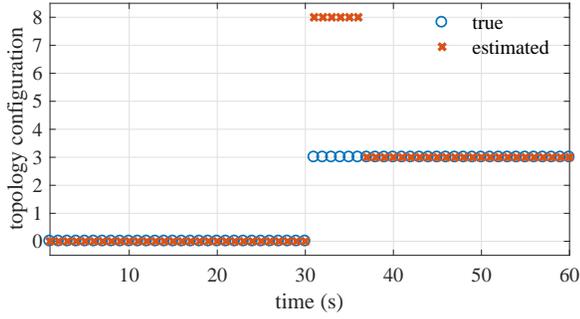}
\caption{Estimated topology configuration under topology reconfiguration.}
\label{fig:noisy_reconfig_config}
\end{figure}

\begin{figure}[!t]
\centering
\includegraphics[width=3.in]{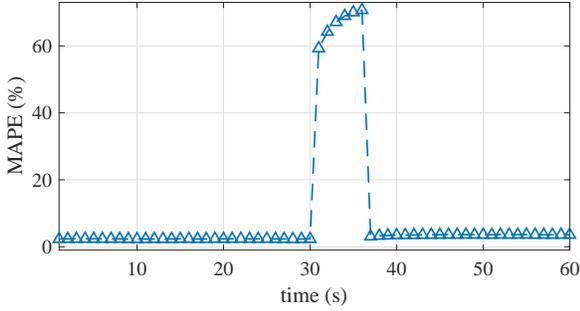}
\caption{MAPE of $\hat{\bm{X}}$ under topology reconfiguration.}
\label{fig:noisy_reconfig_mape_X_est}
\end{figure}

The proposed algorithm works well not only under a fixed topology configuration but also when topology reconfiguration occurs.
To illustrate this, we simulate a case where the underlying topology configuration is changed from configuration $0$ to configuration $3$ at $31$~s.
$60$~sets of measurements are used to compute the voltage sensitivities, i.e., $|\calK| = 60$.
The discount factor $\gamma$ is set to $0.6$.
The minimum residual error and the corresponding estimated topology configuration are shown in Figs. \ref{fig:noisy_reconfig_residual} and \ref{fig:noisy_reconfig_config}, respectively.
A jump in the minimum residual error is observed when the topology is reconfigured.
The new topology is successfully identified after $6$~s.
Correspondingly, the MAPE of $\hat{\bm{X}}$ is also reduced to less than $2\%$ after $6$~s, as is shown in Fig. \ref{fig:noisy_reconfig_mape_X_est}.

\subsection{Voltage Control Performance}

\begin{figure}[!t]
\centering
\includegraphics[width=3.in]{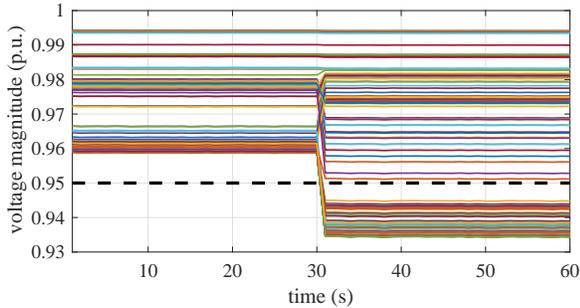}
\caption{Voltage profiles with model-based voltage regulation scheme under topology reconfiguration.}
\label{fig:noisy_reconfig_voltage_model}
\end{figure}

\begin{figure}[!t]
\centering
\includegraphics[width=3.in]{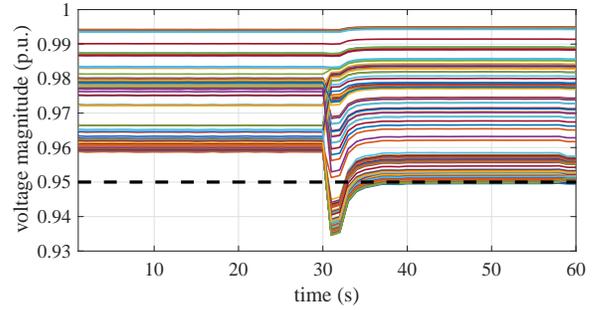}
\caption{Voltage profiles with proposed voltage regulation scheme under topology reconfiguration.}
\label{fig:noisy_reconfig_voltage_data}
\end{figure}

\begin{figure}[!t]
\centering
\includegraphics[width=3.in]{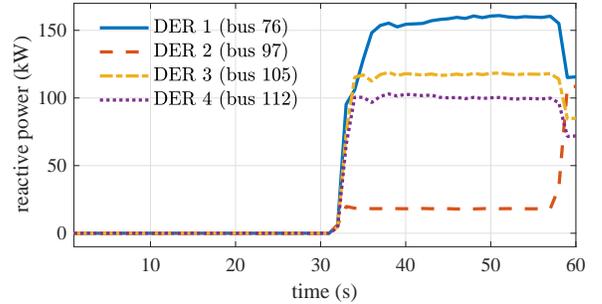}
\caption{DER reactive power injections with proposed voltage regulation under topology reconfiguration.}
\label{fig:noisy_reconfig_reactive_data}
\end{figure}

Next, we show the performance of the proposed voltage regulation framework in the same case as the one in the previous section with topology reconfiguration, where the underlying topology configuration is changed from configuration $0$ to configuration $3$ at $31$~s.
A mode-based voltage regulation scheme, which assumes the true voltage sensitivity matrices are known but is not aware of the topology reconfiguration, is used for the purpose of comparison.
The voltage profiles with the model-based and the proposed voltage regulation schemes are presented in Figs. \ref{fig:noisy_reconfig_voltage_model} and \ref{fig:noisy_reconfig_voltage_data}, respectively, and the DER reactive power injections are shown in Fig. \ref{fig:noisy_reconfig_reactive_data}.
It is obvious that the proposed data-driven voltage regulation framework is very effective and efficient in restoring the voltage magnitudes to the desirable range.
This illustrates the strong adaptivity of our voltage regulation framework to system condition changes such as topology reconfiguration.

%%%%%%%%%%%%%%%%%%%%%%%%%%%%%%%%%%%%%%%%%%%%%%%%
\section{Concluding Remarks} \label{sec:con}

In this paper, we proposed a data-driven voltage regulation framework for DERs in a balanced radial power distribution system.
This framework utilizes a linear model that approximates the nonlinear relation between the voltage magnitudes and power injections, and estimates its parameters---the so-called voltage sensitivities---indirectly by estimating the topology configuration and the corresponding line parameters.
In particular, the proposed estimation algorithm for the voltage sensitivities requires much fewer data than existing ones by exploiting the structural characteristics of the power distribution system.
Using the estimated voltage sensitivities, the optimal DER power injections can be readily determined by solving a convex optimization problem.

Theoretical analysis shows that the voltage sensitivities of interest are easily identifiable.
The inherent data-driven nature of the framework makes it adaptive to changes in system conditions, such as topology reconfigurations.
Numerical simulations illustrated that the voltage sensitivities can be estimated accurately using a few set of measurements even under topology reconfiguration; consequently, guaranteeing good voltage regulation performance.

\bibliographystyle{IEEEtran}
\bibliography{ref}

\end{document}